\documentclass[11pt]{amsart}
\usepackage{amscd,amssymb}
\usepackage{newlfont,amsmath,mathrsfs,enumerate,amssymb,array,amscd,xcolor}

\usepackage{graphicx}
\usepackage[all]{xy}

\theoremstyle{plain}

\newtheorem{prop}[equation]{Proposition}
\newtheorem{cor}[equation]{Corollary}

\newtheorem{lemma}[equation]{Lemma}

\newtheorem{defn}[equation]{Definition}

\numberwithin{equation}{section}

\newcommand{\Z}{\mathbb Z}

\newcommand{\G}{\mathfrak G}

\newcommand{\A}{\mathbb A}

\def\Hom{{\mathrm{Hom}}}
\def\Aut{{\mathrm{Aut}}}

\def\G{{\rm G}}
\def\SL{{\mathrm{SL}}}

\def\GSp{{\mathrm{GSp}}}

\def\PGSp{\mathrm{PGSp}}
\def\PGSO{\mathrm{PGSO}}
\def\Sp{{\mathrm{Sp}}}
\def\Mp{{\mathrm{Mp}}}
\def\Spin{\mathrm{Spin}}
\def\GU{{\mathrm {GU}}}
\def\SU{{\mathrm {SU}}}
\def\U{{\mathrm U}}
\def\GO{{\mathrm GO}}
\def\GL{{\mathrm{GL}}}
\def\PGL{{\mathrm{PGL}}
\def\GSO{\mathrm GSO}}

\def\SO{\mathrm{SO}}
\def\Spin{{\mathrm{Spin}}}
\def\Ind{{\mathrm{Ind}}}
\def\Irr{{\mathrm{Irr}}}
\def\Mp{\mathrm{Mp}}

\def\O{{\mathrm O}}

\def\Res{{\mathrm{Res}}}

\def\der{{\mathrm{der}}}

\def\Ker{{\mathrm{Ker}}}
\def\sim{{\mathrm{sim}}}
\def\ind{{\mathrm{ind}}}

\def\A{{\mathbb A}}

\def\Z{{\mathbb Z}}

\def\G{{\mathbb G}}

\title[Similitude exceptional theta correspondences]{Similitude exceptional theta correspondences}

\author{Petar Baki\'c, Wee Teck Gan and Gordan Savin}

\address{P. B.: Department of Mathematics, University of Utah, Salt Lake City, UT}\email{bakic@math.utah.edu}
\address{W.T.G.:   Department of Mathematics, National University of Singapore, 10 Lower Kent Ridge Road
Singapore 119076} \email{matgwt@nus.edu.sg}
\address{G. S.: Department of Mathematics, University of Utah, Salt Lake City, UT}\email{savin@math.utah.edu}
\dedicatory{In honor of Marko Tadi\'c \\ from whom we have learned so much}

\begin{document} 

\maketitle

\section{\bf Introduction}  \label{S:intro}

Let $F$ be a local field and suppose $(G, H)$ is a reductive dual pair in simple linear algebraic group $\mathcal{E}$ over $F$. 
Hence, $G$ and $H$ are subgroups of $\mathcal{E}$ which are mutual centralizers of each other.
We thus have a homomorphism of algebraic groups
\[  i: G \times H \longrightarrow \mathcal{E}. \]
If $\Omega$ is a minimal representation of $\mathcal{E}(F)$, we may pull back $\Omega$ to the product group $G(F) \times H(F)$ and study the resulting branching problem: this is the usual set up of theta correspondence.

\vskip 5pt

However, the map $i$ is frequently not injective. Typically, one has  (multiplicative type) central subgroups
\[  Z \hookrightarrow G \quad \text{and} \quad  Z  \hookrightarrow H \]
which are identified under $i$. More precisely, one has 
\[ \mathrm{Ker}(i) =   Z^{\nabla} = \{ (z, z^{-1}): z \in Z \},\]
 so that one has an injecion
\[  (G \times H)/  Z^{\nabla} \hookrightarrow \mathcal{E}. \]
This already occurs in the setting of the classical theta correspondence for the symplectic-orthogonal dual pair 
\[  i: \O(V) \times \Sp(W)  \longrightarrow \Sp(V \otimes W) \]
whose kernel is the diagonally embedded $Z= \mu_2$. 
\vskip 5pt

In this short note,  we explain how the non-injectivity of $i$ can be exploited to extend the theta correspondence for $G(F) \times H(F)$ to the setting of similitude groups associated to $G$ and $H$. Indeed, as we shall see, this non-injectivity is the reason for the existence of such a theory.
\vskip 5pt

Here is a brief summary of the content of the paper.  In \S 2, we associate to a dual pair $G \times H \longrightarrow \mathcal{E}$ as above, and some additional data, a similitude dual pair $\tilde{G} \times \tilde{H}$.  After introducing the associated similitude version of theta correspondence in \S 3 and demonstrating some simple properties of this similitude theta correspondence in \S 4, we show in \S 5 that the Howe duality theorem holds for $G \times H$ if and only if the analogous theorem holds for $\tilde{G} \times \tilde{H}$; this is Proposition \ref{P:Howe} and should be considered the main result of this paper. The rather long \S 6 is devoted to 
highlighting a few families of examples in the context of exceptional groups. In \S 7, we discuss how the useful notion of seesaw duality can be extended to the similitude setting, using the families of examples discussed in \S 6 as illustration. Finally, \S 8 discusses the similitude theta correspondences in the global setting and \S 9 collects together some basic results in Clifford theory that are used in the paper. 
\vskip 5pt

\section{\bf Similitude Groups} \label{S:simi group}
Let us first introduce the notion and construction of the relevant similitude groups. We will continue to work in the context of the introduction. 
\vskip 5pt

\subsection{\bf Initial data} \label{SS:ID}
The initial data needed for the construction of similitude groups is described in the following hypotheses: 

\vskip 5pt

\begin{itemize}
\item[(a)]  the commutative group $Z$ can be embedded into an induced torus:
\[    j: Z \hookrightarrow  T: = \mathrm{Res}_{K/F} \mathbb{G}_m \]
where $K$ is an \'etale $F$-algebra of finite rank. 
\vskip 5pt

\item[(b)]  One has
\[  S:= T / j(Z) \hookrightarrow \mathrm{Res}_{E/F}(\mathbb{G}_m), \]
for some $E/F$. 
\end{itemize}
 We shall fix these data in what follows.

\vskip 5pt

 \subsection{\bf Examples} \label{SS:Examples}
 Let us give some examples:

\begin{itemize}
\item[(1)] when $Z = \prod_{i=1}^r \mu_{n_i}$, one takes embeddings $\mu_{n_i} \hookrightarrow \G_m$, so that one has
\[  \begin{CD} 
1 @>>> Z @>>>T =  \prod_{i=1}^r \G_m @>>> T = \prod_{i=1}^r \G_m @>>> 1\end{CD} \]
where the map $T \rightarrow T$ is $z \mapsto z^{n_i}$ on the $i$-th coordinate. 
\vskip 5pt

\item[(2)]  when 
\[  Z = \mathrm{Res}^1_{K/F}(\G_m) = \mathrm{Ker}(N_{K/F} :  \mathrm{Res}_{K/F}(\G_m) \rightarrow \G_m), \]
 one takes $T = \mathrm{Res}_{K/F} \G_m$ and one has
 \[  \begin{CD}
 1 @>>> Z @>>> T @>N_{K/F}>> \G_m @>>> 1 \end{CD} \]
 \vskip 5pt
 
 \item[(3)]  when
 \[  Z = \mathrm{Res}^1_{K/F} (\mu_n)  = \mathrm{Ker}(N_{K/F}: \mathrm{Res}_{K/F}(\mu_n) \rightarrow  \G_m), \]
 one takes $T = \mathrm{Res}_{K/F} \G_m$ and one has
 \[  \begin{CD}
 1 @>>> Z @>>>  T @>N_{K/F} \times [n]>>  \G_m \times  \mathrm{Res}_{K/F}(\G_m).    
 \end{CD} \]
 Note that in this last example, the sequence above is not exact on the right, as the image of the last arrow is a codimension 1 torus. 
 Consider however the special case when 
 \[  |n - [K:F]| = 1. \]
 In this case, it turns out that the composite
 \[  \begin{CD}
  T / Z @>N_{K/F} \times [n]>> \G_m \times  \mathrm{Res}_{K/F}(\G_m) @>(t,x) \mapsto t x^{-1}>>  \mathrm{Res}_{K/F}(\G_m)  = T
  \end{CD} \]
  is an isomorphism of algebraic tori. Indeed, if $x \in  \mathrm{Res}_{K/F}(\G_m)$ lies in the kernel of the above composite map, so that 
  \[  N_{K/F}(x)  =  x^n, \]
  then on taking norms on both sides, we conclude  that
  \[   N_{K/F}(x)^{[K:F]}  = N_{K/F}(x)^n,  \]
  and hence $x^n = N_{K/F}(x) = 1$, i.e. $x \in \mathrm{Res}_{K/F}^1(\mu_n)$. In the examples we shall consider later,  the rather peculiar condition $ |n - [K:F]| = 1$ will turn out to be satisfied.
  \end{itemize}
\vskip 5pt

\subsection{\bf Similitude groups}
Now we make the following definition:
\vskip 5pt

\begin{defn}
The similitude group $\tilde{G}$ associated to $G$ and the data in (a) and (b) above  is:
\[  \tilde{G} := (G \times  T)/ Z^{\nabla}. \]
The associated similitude homomorphism is the map
\[  \sim_G: \tilde{G}   \twoheadrightarrow T/ j(Z) \hookrightarrow \mathrm{Res}_{E/F} \G_m. \]
given by the second projection and (b). 
\end{defn}
\vskip 5pt

Then observe that one has:
\[  \begin{CD}
 1 @>>> G @>>> \tilde{G} @>\sim_G>>  \mathrm{Res}_{E/F} \G_m   \end{CD} \]
 and
 \[  \begin{CD}
 1 @>>> T @>>> \tilde{G} @>>> G/Z @>>> 1 \end{CD} \]
Likewise, we have the similitude group and similitude character 
\[  \sim_H : \tilde{H}  := (H \times T)/ Z^{\nabla} \longrightarrow \mathrm{Res}_{E/F}(\G_m)  \]
associated to $H$ and the data in (a) and (b), as well as analogs  of the two exact sequences above. 
\vskip 5pt

\subsection{\bf Surjectivity of similitude factor}
Recalling that $S = T/ j(Z)$, we have a commutative diagram of short exact sequences of algebraic groups:
\[ \begin{CD}
 1 @>>> Z @>j>> T @>>> S @>>> 1 \\
 @. @VVV  @VVV   @|   \\
 1 @>>> G @>>> \tilde{G} @>\sim_G>>  S @>>> 1.   \end{CD} \]
On taking $F$-points, we obtain the following commutative diagram with exact rows:
\[  \begin{CD}
T(F) @>>> S(F)  @>\alpha>> H^1(F, Z) \\
@VVV @| @VV{\beta}V \\
\tilde{G}(F) @>\sim_G>> S(F)@>>> H^1(F, G)  
\end{CD} \]
From this, we see that the similitude map $\sim_G : \tilde{G}(F) \longrightarrow S(F)$ is surjective if and only if the map $S(F) \longrightarrow H^1(F, G)$ is $0$.
Since the latter map factors as:
\[  \begin{CD} 
S(F) @>\alpha>> H^1(F, Z) @>\beta>> H^1(F, G), 
\end{CD} \]
 we obtain:
 \vskip 5pt
 
 \begin{prop} \label{P:surj}
 The map $\sim_G : \tilde{G}(F) \longrightarrow S(F)$ is surjective if one of  the following conditions holds:
 
 \vskip 5pt
 \begin{itemize}
 \item  $\alpha: S(F) \longrightarrow H^1(F, Z)$ is $0$;
 \item $\beta: H^1(F,Z) \longrightarrow H^1(F, G)$ is $0$.
 \end{itemize}
 \end{prop}
 Observe that the first condition in the proposition depends only on the initial data (a) and (b) in \S \ref{SS:ID}, whereas the second condition depends only on the pair $Z \subset G$.

\vskip 5pt
\subsection{\bf The group $\tilde{J}^{\sim}$}
Going back to dual pairs, let us set
\[   \tilde{J}^{\sim} = (\tilde{G} \times \tilde{H})^{\sim} = \{ (g,h) \in \tilde{G} \times \tilde{H}: \sim_G(g) \cdot \sim_H(h) =1 \}.\]
Note that one has exact sequences:
\[ \begin{CD}
1 @>>> G(F) @>>> \tilde{J}^{\sim}(F)  @>>> \tilde{H}(F)   \end{CD} \]
and
\[  \begin{CD}
 1@>>>  H(F) @>>> \tilde{J}^{\sim}(F) @>>> \tilde{G}(F) \end{CD} \]
One denotes the image of the last arrow in each sequence above by $\tilde{H}(F)^+$ and $\tilde{G}(F)^+$ respectively. Then one has the following containments with finite index:
\[  T(F) \cdot G(F) \subset \tilde{G}(F)^+ \quad \text{and} \quad   T(F)\cdot H(F)  \subset \tilde{H}(F)^+. \]
Observe moreover that 
\[  \tilde{G}(F)^+ = \tilde{G}(F)\quad  \text{if and only if} \quad  \sim_G(\tilde{G}(F))\subseteq \sim_H(\tilde{H}(F)) \]
and
\[  \tilde{H}(F)^+ = \tilde{H}(F) \quad \text{if and only if} \quad  \sim_H(\tilde{H}(F))\subseteq \sim_G(\tilde{G}(F)). \]
\vskip 5pt

\noindent In particular, $\tilde{G}(F)^+ = \tilde{G}(F)$ if $\sim_H: \tilde{H}(F) \longrightarrow S(F)$ is surjective. 
\vskip 5pt

\begin{defn}
We shall call the pair of groups $(\tilde{G}(F)^+, \tilde{H}(F)^+)$ constructed above a similitude dual pair.
\end{defn}

\vskip 5pt

 \vskip 5pt

\section{\bf Similitude Theta Correspondence}  \label{S:simi theta}
We now consider the theory of theta correspondence. 
Suppose that $\Omega$ is a minimal representation of $\mathcal{E}(F)$ \cite{GS}.
If $\pi \in \Irr (G(F))$, then the big theta lift of $\pi$ to $H(F)$ is by definition
\[  \Theta(\pi) = (\Omega \otimes \pi^{\vee})_{G(F)}, \]
which is a smooth representation of $H(F)$. Likewise, one has a smooth $G(F)$-representation $\Theta(\sigma)$ for any $\sigma \in \Irr (H(F))$.
\vskip 5pt

\begin{defn}
 We say that the dual pair $(G,H)$ satisfies the Howe duality property, or the Howe duality theorem holds for $(G,H)$, if $\Theta(\pi)$ is of finite length with a unique irreducible quotient (if nonzero) for any $\pi \in \Irr(G(F))$ and likewise for $\Theta(\sigma)$ for any $\sigma \in \Irr (H(F))$. 
 \end{defn}
 \vskip 5pt
 
 We would now like to extend the theta correspondence for $(G(F),H(F))$ to the setting of the similitude dual pair $(\tilde{G}(F)^+ , \tilde{H}(F)^+)$. For this, 
 we observe  that 
 \[  T^{\nabla}=\{ (t, t^{-1}): t \in T\} \subset \tilde{J}^{\sim} \]
 is a central subgroup. 
 Setting
\[  J^{\sim} = \tilde{J}^{\sim} /  T^{\nabla} = \left(  (\tilde{G} \times \tilde{H})^{\sim} \right) /  T^{\nabla},  \]
one sees that the natural inclusion gives an isomorphism
\[     (G \times H)/ Z^{\nabla} \simeq J^{\sim} = \left(  (\tilde{G} \times \tilde{H})^{\sim} \right) /  T^{\nabla}. \]
In particular, one has a natural homomorphism
\[  \iota: \tilde{J}^{\sim} \twoheadrightarrow J^{\sim} \hookrightarrow \mathcal{E}. \]
On taking $F$-points and noting that $T$ is an induced torus, so that   $H^1(F, T)$ is trivial, we see that
\[   ((G \times H)/ Z^{\nabla})(F) = J^{\sim}(F) =  (\tilde{G}(F) \times \tilde{H}(F))^{\sim} /  T^{\nabla}(F), \]
so that one has
\[  \iota: \tilde{J}^{\sim}(F)  \twoheadrightarrow    J^{\sim}(F)  \hookrightarrow \mathcal{E}(F). \]
In particular,  any representation of $\mathcal{E}(F)$ can be  pulled back via $\iota$ to
\[  \tilde{J}^{\sim}(F) = (\tilde{G}(F) \times \tilde{H}(F))^{\sim}. \]

 \vskip 5pt
 
  Now one has the following short exact sequences:
 \[ \begin{CD}
1 @>>> G(F) @>>> \tilde{J}^{\sim}(F)  @>>> \tilde{H}(F)^+ @>>> 1   \end{CD} \]
and
\[  \begin{CD}
 1@>>>  H(F) @>>> \tilde{J}^{\sim}(F) @>>> \tilde{G}(F)^+ @>>>1.  \end{CD} \]
\vskip 5pt

\noindent Given an irreducible representation $\tilde{\pi}$ of $\tilde{G}(F)^+$, we may regard $\tilde{\pi}^{\vee}$ as a representation of $\tilde{J}^{\sim}(F)$ via the natural surjection $\tilde{J}^{\sim}(F) \rightarrow \tilde{G}(F)^+$ given in the second short exact sequence above. Noting that $G(F)$ is a normal subgroup of $\tilde{J}^{\sim}(F)$ by the first short exact sequence above, we set
 \[  \Theta(\tilde{\pi}) := ( \Omega \otimes \tilde{\pi}^{\vee})_{G(F)}, \]
  so that $\Theta(\tilde{\pi})$ is a smooth representation of $\tilde{H}(F)^+$.
   Likewise for $\tilde{\sigma} \in \Irr (\tilde{H}(F)^+)$, one has
 \[  \Theta(\tilde{\sigma} )=  ( \Omega \otimes \tilde{\sigma}^{\vee})_{H(F)}  \]
 which is a smooth representation of $\tilde{G}(F)^+$.
 \vskip 10pt

 One can rephrase the above definitions slightly by introducing the similitude minimal representation. Set
 \[  \tilde{\Omega} = {\ind}_{\tilde{J}^{\sim}(F)}^{\tilde{J}(F)^+} \Omega \]
 where 
 \[  \tilde{J} = \tilde{G} \times \tilde{H} \]
 and
 \[  \tilde{J}(F)^+ = \tilde{G}(F)^+ \times \tilde{H}(F)^+.\]
  Then  $\tilde{\Omega}$ is a representation of the similitude dual pair $\tilde{G}(F)^+ \times \tilde{H}(F)^+$ and we can defined $\Theta(\tilde{\pi})$ and $\Theta(\tilde{\sigma})$ in the usual way.
 We leave it to the reader to verify that the two definitions are the same.
 \vskip 5pt
   
 As in the ``isometry" case, we can define the notion of the Howe duality property holding for $(\tilde{G}(F)^+, \tilde{H}(F)^+)$: 
 \vskip 5pt
 
 \begin{defn}
 We say that the similitude dual pair $(\tilde{G}(F)^+, \tilde{H}(F)^+)$ satisfies the Howe duality property if 
 $\Theta(\tilde{\pi})$ has finite length and unique irreducible quotient (if nonzero) for any $\tilde{\pi} \in \Irr (\tilde{G}(F)^+)$, and likewise for $\Theta(\tilde{\sigma})$ with $\tilde{\sigma} \in \Irr (\tilde{H}(F)^+)$.
\end{defn}

\vskip 5pt

\section{\bf Central Characters}  \label{S:central}
Let us record some simple properties of the similitude theta correspondence.
\vskip 5pt

 \begin{lemma}[Central characters]  \label{L:central}
 Suppose that $\pi \in \Irr( \tilde{G}(F)^+)$ has  $T(F)$-central character $\chi$. Then the smooth $\tilde{H}(F)^+$-module $\Theta(\tilde{\pi})$ 
has $T(F)$-central character $\chi$. 
 \end{lemma}
\begin{proof}
This is because $ T^{\nabla}(F) \subset \tilde{J}^{\sim}(F)$  acts trivially on $\Omega$ and so $T(F) \subset \tilde{H}(F)^+$ acts on
\[  \Theta(\tilde{\pi}) = (\Omega \otimes \tilde{\pi}^{\vee})_{G(F)} \] 
in the same way as $T(F) \subset \tilde{G}(F)^+$ acts on $\tilde{\pi}$. 
\end{proof}

One can also decompose the $\tilde{J}(F)^+$-module  $\tilde{\Omega}$  according to characters of the central subgroup $T(F) \times T(F)/  T^{\nabla}(F)$.
Each character of $T(F) \times T(F) /  T^{\nabla}(F)$ is of the form $\chi \otimes \chi$, where $\chi$ is a character of $T(F)$. 
We can consider the $\chi$-isotypic quotient
$\tilde{\Omega}_{T,\chi}$ of $\tilde{\Omega}$.  which can be described as follows.
\vskip 5pt

One has an intermediate group
\[  \tilde{J}^{\sim}(F)  \subset \tilde{J}^{\sim}(F) \cdot (T(F) \times T(F)) \subset \tilde{J}(F)^+. \]
We may extend the $\tilde{J}^{\sim}(F)$-module $\Omega$ to the intermediate group $\tilde{J}^{\sim}(F) \cdot (T(F) \times T(F))$ by letting $T(F) \times T(F)$ acts by $\chi \otimes \chi$.
This extension is well-defined because $\chi \otimes \chi$ is trivial on the intersection
\[  \tilde{J}^{\sim}(F) \cap (T(F) \times T(F)) =  T^{\nabla}(F). \]
We denote this extended representation by $\Omega_{\chi}$
 Then one has
\[  \tilde{\Omega}_{T, \chi} \simeq  \ind_{\tilde{J}^{\sim}(F) \cdot (T(F) \times T(F))}^{\tilde{J}(F)^+}  \Omega_{\chi}. \]
It will sometimes be convenient to restrict to a fixed $T(F)$-central character $\chi$, in which case one will be working with $\tilde{\Omega}_{T,\chi}$.
\vskip 5pt

\section{\bf The Howe Duality Property}  \label{S:howe}
In this section, we shall show:
\vskip 5pt

\begin{prop} \label{P:Howe}
The Howe duality property holds for the  dual pair $(G(F), H(F))$ if and only if the Howe duality property holds for the similitude dual pair $(\tilde{G}(F)^+, \tilde{H}(F)^+)$.
\end{prop}

\subsection{\bf Some lemmas}  \label{SS:lemmas}
We shall show this in a series of lemmas, beginning with the following observation:
\vskip 5pt

\begin{lemma} \label{L:element}
Let $V$ be a not-necessarily-smooth representation of $\tilde{G}(F)^+$ on which the central subgroup $T(F)$ acts by a character $\chi$. Then:
\vskip 5pt

(i)  A vector $v \in V$ a smooth with respect to $\tilde{G}(F)^+$ if and only if it is smooth with respect to $G(F)$. 
\vskip 5pt

(ii) If $V^{\infty}$ denotes the subspace of smooth vectors of $V$ (with respect to $\tilde{G}(F)^+$ or $G(F)$), then $V^{\infty}$ has finite length as a $\tilde{G}(F)^+$-module if and only if it has finite length as a $G(F)$-module.
\vskip 5pt

(iii) If $V^{\infty}$ is an irreducible smooth $\tilde{G}(F)^+$-module, then $V^{\infty}$ is semisimple as a $G(F)$-module. Indeed,
\[  V^{\infty}|_{G(F)} \simeq m \cdot \bigoplus_{i=1}^k V_i \]
where the $V_i$'s are   irreducible smooth $G(F)$-modules which are pairwise inequivalent. Moreover,  $\tilde{G}(F)^+$ permutes the isomorphism classes of the $V_i$'s transitively. 
\vskip 5pt

(iv) Suppose that $V= V^{\infty}$ is of finite length (as $\tilde{G}(F)^+$-module or $G(F)$-module). Let $\tilde{cosoc}(V)$ be the cosocle (or maximal semisimple quotient) of $V$ as a $\tilde{G}(F)^+$-module and likewise let $cosoc(V)$ be the cosocle of $V$ as a $G(F)$-module. Then the natural map 
\[ cosoc(V)   \longrightarrow   \tilde{cosoc}(V)    \]
is an isomorphism of $G(F)$-modules.
\end{lemma}

\begin{proof}
All these follows readily from the fact  that $T(F) \cdot G(F)$  is an open subgroup of finite index in $\tilde{G}(F)^+$ and the Clifford theory recounted in Appendix A below.
More precisely, 
\vskip 5pt
\begin{itemize}
\item (i) follows from the fact that an open compact subgroup of $G(F) \cdot T(F)$ is an open compact subgroup of $\tilde{G}(F)^+$;
\item (ii) follows from Proposition \ref{P:A1}(1); 
\item (iii) follows from Proposition \ref{P:A1}(2) and Proposition \ref{P:A2};
\item (iv) follows from Proposition \ref{P:A1}(3).
\end{itemize}
\end{proof}
\vskip 5pt



\vskip 5pt
Next we have:
\vskip 5pt

\begin{lemma}  \label{L:bigtheta}
 Let  $\tilde{\pi} \in \Irr (\tilde{G}(F)^+)$ and suppose (as in Lemma \ref{L:element}(iii)) that
\[  \tilde{\pi}|_{G(F)} \simeq m \cdot \bigoplus_{i=1}^k \pi_i \]
for some $\pi_i \in \Irr (G(F))$.  Then we have:
\begin{equation} \label{E:dualtheta}
  \Theta(\tilde{\pi})|_{H(F)}  =  m \cdot \bigoplus_i \Theta(\pi_i). \end{equation}
In particular, $\Theta(\tilde{\pi})$ is nonzero if and only if $\Theta(\pi_i) \ne 0$ for some (equivalently for all) $i$.
\end{lemma}

\begin{proof}
We compute:
\begin{align}
\Theta(\tilde{\pi})|_{H(F)} &= \left( \Omega \otimes \tilde{\pi}^{\vee} \right)_{G(F)}  \quad \text{(by definition)}  \notag \\
&= \left( \Omega \otimes ( m \cdot \bigoplus_i \pi_i^{\vee} ) \right)_{G(F)} \notag \\
&= m \cdot \bigoplus_i  (\Omega \otimes \pi_i^{\vee})_{G(F)}  \notag \\
&= m \cdot \bigoplus_i \Theta(\pi_i). \notag
\end{align}
Here, we have used Lemma \ref{L:element}(i) which ensures that the contragredient of $\tilde{\pi}$ as a representation of $\tilde{G}(F)$ is the same as its contragredient as a representation of $G(F)$, so that
\[  \tilde{\pi}^{\vee}|_{G(F)}  \simeq m \cdot \bigoplus_i \pi_i^{\vee}. \]
This completes the proof of the lemma.
 
\end{proof}
\vskip 5pt

By Lemma \ref{L:bigtheta} and Lemma \ref{L:element}, we deduce:
\begin{cor}  \label{C:finite}
\[  \text{$\Theta(\tilde{\pi})$ has finite length as an $\tilde{H}(F)^+$-module}  \]
if and only if
\[ \text{$\Theta(\pi_i)$ has finite length as an $H(F)$-module for all $i$.} \]
 Assuming this finiteness condition   holds, one has:
 \[  
\Theta(\tilde{\pi})  =  m \cdot \bigoplus_i \Theta(\pi_i) \]
 and
 \begin{equation} \label{E:small}
  \theta(\tilde{\pi})  =  m \cdot \bigoplus_{i\in  I} \theta(\pi_i) \end{equation}
as $H(F)$-modules.  
\end{cor}
 \vskip 5pt

 \subsection{\bf Proof of Proposition \ref{P:Howe}}  \label{SS:proof}
 We can now prove Proposition \ref{P:Howe}.
 Suppose first that the Howe duality property holds for the pair $(G,H)$. By symmetry,  for any  $\tilde{\pi} \in \Irr (\tilde{G}(F)^+)$, we need to show that if $\Theta(\tilde{\pi})$ is nonzero, then it has finite length and $\theta(\tilde{\pi})$ is irreducible. 
 \vskip 5pt
 
 Suppose that
  \[  \tilde{\pi}|_{G(F)} \simeq m \cdot \bigoplus_{i \in I}  \pi_i \]
 for $\pi_i \in \Irr (G(F))$.  Then by Corollary \ref{C:finite},
 \[  
\Theta(\tilde{\pi})|_{H(F)}  =  m \cdot \bigoplus_i \Theta(\pi_i). \]
By hypothesis, $\Theta(\pi_i)$ has finite length for each $i$, and hence so does $\Theta(\tilde{\pi})$ as a $\tilde{H}(F)$-module by Corollary \ref{C:finite}.  Moreover,
 by Lemma \ref{L:bigtheta}, $\Theta(\pi_i)$ is nonzero for each $i$ and $\sigma_i:= \theta(\pi_i)$ is irreducible by hypothesis. 
 
 \vskip 5pt
   
 Let $\tilde{\sigma} \subset \theta(\tilde{\pi})$ be an irreducible summand. 
 Since (by (\ref{E:small}))
 \[  \theta(\tilde{\pi})|_{H(F)}  =  m \cdot \bigoplus_{i\in  I} \theta(\pi_i)  = m \cdot \bigoplus_{i\in  I} \sigma_i, \]
we see that 
\[  \tilde{\sigma}|_{H(F)} \simeq n \cdot \bigoplus_i \sigma_i \quad \text{for some $0 < n \leq m$.} \]
If $m=1$, then $\theta(\tilde{\pi}) = \tilde{\sigma}$ and we would have been done already. For general $m$, we argue as follows.
Switching the roles of $G$ and $H$ in the above argument, it follows by Corollary \ref{C:finite} that
\[  \theta(\tilde{\sigma})|_{G(F)} = n \cdot \bigoplus_i   \theta(\sigma_i). \]
By hypothesis, each $\theta(\sigma_i)$ is irreducible and hence $\theta(\sigma_i) = \pi_i$.
But since $\tilde{\pi} \subset \theta(\tilde{\sigma})$, we deduce that $m \leq n$.  
Hence, we must have $m = n$, so that $\theta(\tilde{\pi}) = \tilde{\sigma}$ is irreducible.
\vskip 5pt

Conversely, suppose now the that Howe duality property holds for the similitude pair $(\tilde{G}(F)^+, \tilde{H}(F)^+)$. 
By symmetry, take  any $\pi \in \Irr (G(F))$ and choose $\tilde{\pi} \in \Irr( \tilde{G}(F)^+)$ such that $\pi \subset \tilde{\pi}$. Suppose
\[  \tilde{\pi}|_{G(F)} \simeq m \bigoplus_{i \in I} \pi_i \]
with $\pi_i \in \Irr (G(F))$ for each $i$, and  $\pi = \pi_{i_0}$ for some $i_0 \in I$.
\vskip 5pt

By  Cor. \ref{C:finite}, we have
 \[  
\Theta(\tilde{\pi})|_{H(F)}  =  m \cdot \bigoplus_i \Theta(\pi_i). \]
Since we are assuming that $\Theta(\pi) \ne 0$, it follows by Lemma \ref{L:bigtheta} that $\Theta(\tilde{\pi})$ is nonzero.
By hypothesis and Cor. \ref{C:finite}, $\Theta(\tilde{\pi})$ has finite length as a $H(F)$-module, and hence so does $\Theta(\pi_i)$ for each $i$. 
\vskip 5pt

Consider $\tilde{\sigma} := \theta(\tilde{\pi})$, which  is irreducible by hypothesis. By (\ref{E:small}), we see that 
\begin{equation} \label{E:sigma1}
 \tilde{\sigma}|_{H(F)} \simeq  m \cdot \bigoplus_{i \in I} \theta(\pi_i), \end{equation}
with each $\theta(\pi_i)$ semisimple of finite length. On writing each $\theta(\pi_i)$ as a sum of irreducible summands,  we have:
\begin{equation} \label{E:sigma2}
  \tilde{\sigma}|_{H(F)} \simeq  n \bigoplus_{j \in J} \sigma_j  \end{equation}
with 
\[  \text{distinct $\sigma_j \in \Irr( H(F))$},  \quad \text{$|I| \leq |J|$}  \quad \text{and $n \geq m$.}  \]
Then applying (\ref{E:small}) with the roles of $G$ and $H$ exchanged, one has
\[ m \cdot \bigoplus_{i \in I} \pi_i =  \tilde{\pi} |_{G(F)}=  \theta(\tilde{\sigma})|_{G(F)} = n \cdot \bigoplus_{j \in J}  \theta(\sigma_j), \]
with $\theta(\sigma_j) \ne 0$.
 This implies that
  \[  \text{$m = n$,} \quad \text{ $|I| = |J|$}  \quad \text{and $\theta(\sigma_j) = \pi_j$,} \]
  after fixing some bijection of $I$ with $J$.
   In particular, going back to (\ref{E:sigma1}) and (\ref{E:sigma2}), we deduce that $\theta(\pi_i) = \sigma_i$ is irreducible for each $i \in I$.
\vskip 5pt

 This completes the proof of Proposition \ref{P:Howe}.
\vskip 10pt

\section{\bf Examples}   \label{S:examples}
We give some interesting examples of similitude dual pairs and similitude theta correspondences.
\vskip 5pt

\subsection{\bf Classical dual pairs}  \label{SS:classical}
We begin by revisiting the case of classical dual pairs.
\vskip 5pt

\begin{itemize}
\item (Symplectic-orthogonal) For a quadratic space $V$  and a symplectic vector space $W$, have
\[  \iota:  H \times G = \O(V) \times \Sp(W) \longrightarrow \Sp(V \otimes W). \]
Then
\[  \Ker(\iota) = \mu_2^{\nabla}. \]
Hence, we have $Z = \mu_2$ and may take $T = \G_m$, so that
\[  \tilde{H} = \GO(V) \quad \text{and} \quad \tilde{G} = \GSp(W)\]
are the usual similitude groups.
Then one has
\[  \iota:  \tilde{J}^{\sim} = (\GO(V) \times \GSp(W))^{\sim} \longrightarrow \Sp(V \otimes W) \]
extending $i$. 
\vskip 5pt

Recall  however that the Weil representation $\Omega$ is not a representation of $\Sp(V \otimes W)$ but rather of its metaplectic cover $\Mp(V \otimes W)$. In the classical theta correspondence, one needs to first construct liftings of $\iota$ to $\Mp(V \otimes W)$:
\[  \tilde{\iota}: \O(V) \times \Sp(W) \longrightarrow \Mp(V \otimes W), \]
before one can restrict the Weil representation to $\O(V) \times \Sp(W)$. Such splittings exist if $\dim V$ is even and have been systematically constructed by  Kudla. 
Likewise, to obtain a similitude theta correspondence, we would need to extend $\tilde{\iota}$ to:
\[  \tilde{\iota}: (\GO(V) \times \GSp(W))^{\sim} \longrightarrow \Mp(V \otimes W). \]
Such extensions have been constructed by B. Roberts \cite{R}. This accounts for the main complexity in the theory of similitude theta correspondences for symplectic-orthogonal  dual pairs.

\vskip 5pt

\item (Unitary) If $V$ is a Hermtitian space and $W$ a skew-Hermitian space relative to a quadratic extension $E/F$, one has
\[  \iota: \U(V) \times \U(W) \longrightarrow \Sp(\Res_{E/F} (V \otimes_E W))\]
with
\[  \Ker(\iota) = ( E^1)^{\nabla} \]
where we have written $E^1$ for  $\mathrm{Res}_{E/F}^1(\G_m)$. 
In this case, we take $T= \Res_{E/F} \G_m$ to get the usual similitude groups
\[  \tilde{G} = \GU(V) \quad \text{and} \quad \tilde{H} = \GU(W). \]
As in the symplectic-orthogonal case, one needs to construct a lifting 
\[  \tilde{\iota} : \tilde{J}^{\sim} = (\GU(V) \times \GU(W))^{\sim} \longrightarrow \Mp( \Res_{E/F}(V \otimes_E W)) \]
before one can consider the similitude theta correspondence. However, the unitary case is somewhat better than the symplectic-orthogonal one. 

 \vskip 5pt
 
 More precisely, note that 
 \[ \iota: \tilde{J}^{\sim} = (\GU(V) \times \GU(W))^{\sim} \longrightarrow \Sp(V \otimes W) \]
 has image contained in $\U(V \otimes_E W)$. However, the metaplectic covering is split over $\U(V \otimes_E W)$! By fixing such a splitting, we thus obtain
 \[  \tilde{J}^{\sim} \longrightarrow \U(V \otimes_E W) \longrightarrow \Mp(V \otimes_E W). \]
  Thus, there is no need to do extra work beyond that needed for splitting the isometry dual pairs.
\end{itemize} 
 \vskip 5pt
 
 \subsection{\bf Exceptional dual pairs}  \label{SS:exceptional}
 Given the somewhat sporadic nature of the geometry of exceptional groups, it is not surprising that it is harder to formulate a uniform theory of dual pairs in exceptional groups.
 Nonetheless, we shall attempt to do so for a few families of such dual pairs, by realizing the dual pairs as
 \[  G \times H = \Aut(A) \times \Aut(B) \]
 for two algebraic structures $A$ and $B$. 
 
 \vskip 5pt

 \subsection{\bf Composition algebras and cubic  norm structures}
 We refer to reader to \cite[Chap. VIII, \S 33]{KMRT} for the notion of composition algebras and \cite[Chap. 9, \S 37 and \S38]{KMRT} for the notion of cubic norm structures (also known as Freuthendal-Jordan algebras).
 Let $C$ be a composition algebra over $F$, so that $\dim C = 1$, $2$, $4$ or $8$. Let $J$ be a nontrivial cubic norm structure over $F$, so that $J$ has dimension $3$, $6$, $9$, $15$ or $27$.     Then one has an isometry dual pair \cite{MS, Ru}
 \[  \Aut(C) \times \Aut(J)  \hookrightarrow   \mathcal{E} \]
 where $\mathcal{E}$ is a certain ambient  adjoint group. We enumerate the most interesting  (split) cases, where $C$ or $J$ have the maximal dimensions. 
 \vskip 5pt
 
 \begin{itemize}
 \item[-]  For $\dim C = 8$, one has $\Aut(C) = G_2$ and as $J$ varies, one has the dual pairs 
 \[G_2 \times \Aut(J)  \subset \mathcal{E}  \]
 with $\Aut(J)$ and $\mathcal{E}$ given in the following table.
 
 \vskip 5pt
 \begin{center}
 \begin{tabular}{|c|c|c|c|c|c|}
 \hline
 $\dim J$  &  3 & 6 & 9 & 15 & 27  \\
 \hline
 $\Aut(J)$ & $S_3$ & $\SO_3$ & $\PGL_3 \rtimes \Z/2\Z$ & $\PGSp_6$ & $F_4$  \\
 \hline
  $\mathcal{E}$ & $\PGSO_8 \rtimes S_3$ & $F_4$ & $E_6 \rtimes \Z/2\Z$ & $E_7$ & $E_8$ \\
  \hline
  \end{tabular}
 \end{center}

 \vskip 10pt
 
 \item[-] For $\dim J = 27$, one has $\Aut(J) = F_4$ and as $C$ varies, one has the dual pairs
 \[   F_4 \times \Aut(C)  \subset \mathcal{E} \]
  with $\Aut(C)$ and $\mathcal{E}$ given in the following table.
  \vskip 5pt
  \begin{center}
 \begin{tabular}{|c|c|c|c|}
 \hline
 $\dim C$  &  2 & 4 & 8   \\
 \hline
$\Aut(C)$ & $
Z/2\Z$ & $\PGL_2$ & $G_2$ \\
\hline
$\mathcal{E}$ & $E_6 \rtimes \Z/2\Z$ & $E_7$ & $E_8$ \\
\hline
\end{tabular}
\end{center}
 \end{itemize}
 \vskip 5pt
 
   Observe that the map from $\Aut(C) \times \Aut(J)$ to $\mathcal{E}$ is injective. Thus, in such cases, we do not have a theory of similitude theta correspondence.
 \vskip 10pt

 \subsection{\bf Twisted composition algebras}
 Let $E$ be an \'etale cubic $F$-algebra. Then one has the notion of twisted composition algebras with respect to $E/F$; the reader can consult  \cite[Chap. 8, \S 36]{KMRT} for this.
   One way such a twisted composition algebra arises is to start with a composition $F$-algebra $C$  and consider $C^{\flat} = C \otimes_F E$. Then a construction in \cite[\S 36C, Pg 499]{KMRT} equips $C^{\flat}$ with the structure of a twisted composition algebra, built out of the composition algebra structure on $C$.
 In fact, if $C^{\flat}$ is any  twisted composition algebra, then $\dim_E(C^{\flat}) = 1$, $2$, $4$ or $8$ \cite[Cor. 36.4, Pg 492]{KMRT}.
 \vskip 5pt
 
 Let $C^{\flat}_1$ and $C^{\flat}_2$ be two twisted composition algebras relative to $E/F$. Then one has a dual pair \cite{GS2}
 \[ i:   \Aut_E(C^{\flat}_1) \times \Aut_E(C^{\flat}_2)  \longrightarrow \mathcal{E} \]
 where $\mathcal{E}$ is some ambient adjoint group. We enumerate the most interesting case with $\dim_E C^{\flat}_1$ maximal. Then one has $\Aut_E(C^{\flat}_1) =\Spin_8^E$, a simply-connected quasi-split group of type $D_4$ determined by $E$.  As $C^{\flat}$ varies, one has the dual pair 
 \[   \Spin_8^E \times \Aut_E(C^{\flat}) \longrightarrow \mathcal{E} \]
 with $\Aut(C^{\flat})$ and $\mathcal{E}$ given in the following table (assuming $E = F^3$ so that the groups are split).
 \vskip 5pt
 
 \begin{center}
 \begin{tabular}{|c|c|c|c|c|}
 \hline
 $\dim_E C^{\flat}$ & 1  &  2 & 4 & 8   \\
 \hline
$\Aut(C^{\flat})$ & $(\Res_{E/F} \mu_2 )/ \mu_2$ & $  (\Res_{E/F}\G_m)/ \G_m \rtimes \Z/2\Z$ & $\Res_{E/F}(\SL_2)/ \mu_2$ & $\Spin^E_8$ \\
\hline
$\mathcal{E}$ & $F_4$ & $E_6 \rtimes \Z/2\Z$  & $E_7$ & $E_8$ \\
\hline
\end{tabular}
\end{center}
 \vskip 5pt
 
 Observe that the group $\Aut_E(C^{\flat})$ has center isomorphic to
 \[  Z = \Res_{E/F}(\mu_2) / \mu_2 \simeq  \Res^1_{E/F}(\mu_2) \]
 where the last isomorphism is given by
 \[   x \mapsto N_{E/F}(x) /x =: x^{\#}. \]
 Moreover, the map 
 \[  i:  \Spin_8^E  \times \Aut(C^{\flat}) \longrightarrow \mathcal{E} \]
 has
 \[  {\mathrm Ker}(i) =  Z^{\nabla}  \simeq \mathrm{Res}^1_{E/F}(\mu_2).  \]
 Thus, in this case, we have a theory of similitude theta correspondence which we shall now explicate.
 \vskip 5pt
 
 We first pick the data (a) and (b)  in \S \ref{SS:ID} as in Example (3) of \S \ref{SS:Examples}. Namely,
 we  take $T = \mathrm{Res}_{E/F} (\mathbb{G}_m)$ and consider the natural embedding
 \[  j: \mathrm{Res}^1_{E/F}(\mu_2) \hookrightarrow T = \mathrm{Res}_{E/F} (\mathbb{G}_m). \]
 As in \S \ref{SS:Examples}, we then  have the embedding
 \[ \begin{CD}
 T/ j(Z)  @>N_{E/F} \times [2]>>  \mathbb{G}_m \times \mathrm{Res}_{E/F}(\mathbb{G}_m) 
 \end{CD} \]
 whose cokernel has dimension $1$. In this special case, since $3-2=1$,  it turns out that  
 \[  T/ j(Z) \simeq   \mathrm{Res}_{E/F}(\mathbb{G}_m) = T \]
 via the map
 \[  x \mapsto N_{E/F}(x)/x^2 = x^{\#}/ x. \]
 Now let us consider the similitude groups
 \[ \begin{CD}
  \tilde{G} = (\Spin^E_8 \times T) /Z^{\nabla}  @>pr_2>> T/ j(Z) @>x \mapsto x^{\#}/x>> T = \mathrm{Res}_{E/F}(\mathbb{G}_m) \end{CD} \]
  and
  \[ \begin{CD} 
 \tilde{H} = (\Aut_E(C^{\flat}) \times T)/ Z^{\nabla}  @>pr_2>> T/j(Z) @>x \mapsto x^{\#}/x>> T = \mathrm{Res}_{E/F}(\mathbb{G}_m), \end{CD} \]
 Hence, the similitude maps $\sim_G$ and $\sim_H$ on the group of $F$-rational points both take value in $T(F) = E^{\times}$. 
 \vskip 5pt
  
 We would like to know what the groups $\tilde{G}(F)^+$ and $\tilde{H}(F)^+$ are. 
 If $F$ is nonarchimedean,  then $H^1(F, \Spin_8^E) = 0$ and by Proposition \ref{P:surj}, we see that the map $\sim_G: \tilde{G}(F) \longrightarrow E^{\times}$ is surjective. Hence $\tilde{H}(F)^+ = \tilde{H}(F)$ when $F$ is nonarchimedean.  The same is true for $\tilde{G}(F)^+$ if $\dim_E C^{\flat} = 8$.
We shall now consider the question of whether $\sim_H: \tilde{H}(F) \longrightarrow E^{\times}$ is surjective when $\dim_E C^{\flat} = 2$ or $4$, ignoring the somewhat degenerate case when $\dim_E C^{\flat} =1$. For this, it is convenient to give an alternative description of $\tilde{H}$ for which it is easier to describe $\tilde{H}(F)$. 
 \vskip 5pt
 
 \vskip 5pt
 \begin{itemize}
 \item When $\dim_E C^{\flat} = 4$,  the natural inclusion
\[  \mathrm{Res}_{E/F}(\SL_2 \times \mathbb{G}_m) \hookrightarrow \mathrm{Res}_{E/F}(\GL_2 \times \mathbb{G}_m) \]
descends to an isomorphism
\[  \tilde{H}  \longrightarrow  \mathrm{Res}_{E/F}(\GL_2 \times \mathbb{G}_m)/  \iota(\mathrm{Res}(\mathbb{G}_m)) \]
where
\[  \iota(t) =  (t, (t^{-1})^{\#}). \]
  Then the group $\tilde{H}(F)$ is given by
  \[  
  \tilde{H}(F) = (\GL_2(E) \times E^{\times})/ \iota(E^{\times})   \]
  and the similitude character  is
  \[  \sim_H(h, x) = \det(h) \cdot \frac{x^{\#}}{x}. \] 
 From this, it is clear that $\sim_H$ is surjective, so that $\tilde{G}(F)^+ = G(F)$.   
 \vskip 5pt
 
 \vskip 5pt
 
 \item When $\dim_E C^{\flat} = 2$ and  recalling that $T = \mathrm{Res}_{E/F}(\mathbb{G}_m)$, we have an isomorphism:
 \[ \begin{CD}
  \tilde{H}^0 = (T \times   T) /(\mathbb{G}_m \times \mathrm{Res}_{E/F} (\mu_2))  \\
  @V{\simeq}V(t_1,t_2) \mapsto (t_1,t_1^{-1},t_2)V  \\
  (T \times T \times T)/ (\G_m \times T) 
 \end{CD} \]
where
\begin{itemize}
\item[(i)]  the embedding 
\[ \mathbb{G}_m \times \mathrm{Res}_{E/F} (\mu_2) \longrightarrow T \times T \]
in the first line  is given by
\[  (a, z) \mapsto (az, (z^{-1})^{\#}), \]
 \item[(ii)]  the embedding 
\[  \G_m \times T \longrightarrow T \times T \times T \]
in the second  line is given by
\[  (a,t) \mapsto (at, a^{-1}t, (t^{-1})^{\#}). \]
\end{itemize}
 With this alternative description of $\tilde{H}^0$, all tori involved are induced tori, so that
 \[  \tilde{H}^0(F) \simeq
  (E^{\times} \times E^{\times}  \times E^{\times} )/ (F^{\times} \times E^{\times})  \] 
 where $F^{\times} \times E^{\times}$ is embedded into $(E^{\times})^3$ as in (ii) above.
 In this incarnation, the similitude character is given by
 \[  \sim_H (x,y,z) =  x \cdot y \cdot\frac{z^{\#}}{z} \in E^{\times}.  \]
 In particular, $\sim_H$ is surjective onto $E^{\times}$, so that $\tilde{G}(F)^+ = \tilde{G}(F)$ in this case as well.
 \end{itemize}
 
 \vskip 5pt
 
 In particular, let us consider the case where the central character is trivial. Then we obtain for example the similitude dual pairs of adjoint type:
  \[  \PGSO_8^E(F) \times \begin{cases} 
 \PGSO_8^E(F) \text{ in $E_8$;} \\
 \PGL_2(E) \text{  in $E_7$;} \\
  (E^{\times}/F^{\times}) \rtimes \Z/2\Z \text{ in $E_6 \rtimes \Z/2\Z$.} 
  \end{cases} \]

  \vskip 10pt
 
 \subsection{\bf Jordan pairs}
We refer the reader to \cite[Introduction and Chap. 1]{Loos} for the notion of Jordan pairs. Since this notion  is perhaps less familiar to the reader than the notion of Jordan algebras, let us give some motivation and a brief introduction.  
\vskip 5pt

 A Freuthendal-Jordan algebra or a cubic norm structure $J$ comes equipped with a norm form, which is a cubic form $N_J:  J \longrightarrow F$.  The  similitude group of this cubic form
 \[  \sim(J,N_J) = \{ (g, t) \in \GL(J) \times \G_m: N_J \circ g = t \cdot N_J \} \]
 is  called the structure group of $J$, and we will call the isometry group $\mathrm{iso}(J,N_J)$ of $N_J$ the reduced or special structure group (this consists of those pairs $(g,t)$ with $t=1$).   If $J$ is the 27-dimensional Jordan algebra, for example, $\sim(J, N_J)$ is the group $GE_6$ and ${\mathrm{iso}}(J,N_J)$ is the simply-connected $E_6$. 
 \vskip 5pt
 
 Now the group $\sim(J,N_J)$ acts  irreducibly on $J$ and its linear dual $J^*$, but these two representations are not isomorphic (as their central characters are  different) and there is no reason to favour one of these representation over the other (a more familiar example is: a group isomorphic to $\GL(V)$ has two standard representations).   Indeed, the trace bilinear form on $J$ allows us to identify $J$ with  $J^*$, and this defines an outer automorphism of $\sim(J,N_J)$ which is the inverse map on the center of $\sim(J,N_J)$ and which interchanges the two representations. In addition, in the context of $E_6$, when one considers a quasi-split $E_6$ associated to a quadratic field extension $K/F$, the two 27-dimensional representations are fused together to give a single rational representation over $F$. Hence, such a quasi-split $E_6$ cannot be realized as the isometry group of a cubic form over $F$, unlike the split form. 
 \vskip 5pt

 The theory of Jordan pairs, introduced by Loos \cite{Loos}, treats both these representations $J$ and $J^*$ on equal footing and realizes $\sim(J,N_J)$ as the automorphism group of an algebraic structure on the pair $\{ J, J^*\}$. More formally, a Jordan pair over $F$ consists of the data:
 \vskip 5pt
 \begin{itemize}
 \item a pair $(J^+, J^-)$ of $F$-vector spaces;
 \item a pair of quadratic maps defined over $F$: 
 \[  Q^+: J^+ \longrightarrow \Hom_F(J^-, J^+) \quad \text{and} \quad  Q^- : J^- \longrightarrow \Hom_F(J^+, J^-), \]
  \end{itemize}
 satisfying the following axioms \cite[Pg. 1, Def. 1.2]{Loos} for $\epsilon = \pm$ and any $F$-algebra $K$:
 \vskip 5pt
 \begin{itemize}
 \item[(JP1)]  
 \[  \{ x, y, Q^{\epsilon}(x)(z) \}^{\epsilon} = Q^{\epsilon}(x) \left( \{ y ,x ,z \}^{-\epsilon} \right) \quad \text{for any $x \in J^{\epsilon}(K)$ and $y,z \in J^{-\epsilon}(K)$}; \]
 \item[(JP2)] 
 \[  \{ Q^{\epsilon}(x)(y), y,z \}^{\epsilon}  = \{ x, Q^{-\epsilon}(y)(x), z \} ^{\epsilon} \quad \text{for any $x, z \in J^{\epsilon}(K)$ and $y\in J^{-\epsilon}(K)$}; \]
 \item[(JP3)] 
 \[ Q^{\epsilon}(Q^{\epsilon}(x)(y)) = Q^{\epsilon}(x) \circ Q^{-\epsilon}(y) \circ Q^{\epsilon}(x) \quad \text{for any $x \in J^{\epsilon}(K)$ and $y \in J^{-\epsilon}(K)$}, \]
  \end{itemize}
  where we have set
  \[  \{ x,y,z \}^{\epsilon} = Q^{\epsilon}(x+z)(y) - Q^{\epsilon}(x)(y) - Q^{\epsilon}(z)(y) \in J^{\epsilon} \]
  for the linearization  of the map $(x,y) \mapsto Q^{\epsilon}(x)(y)$.  These axioms may look a bit unwieldy for the uninitiated (including ourselves), but we will not seriously make use of them in this paper. 
 \vskip 5pt
 
  A homomorphism from one Jordan pair $(J^+, J^-)$ to another $(V^+, V^-)$ is a pair of linear maps 
 \[  \phi^{\epsilon}: J^{\epsilon}  \longrightarrow V^{\epsilon} \quad \text{(with $\epsilon = \pm$)} \]
 such that
 \[  \phi^{\epsilon}(Q_J^{\epsilon}(x) (y))  =  Q_V^{\epsilon}(\phi^{\epsilon}(x))( \phi^{-\epsilon}(y)) \quad \text{ for $x \in J^{\epsilon}$ and  $y \in J^{-\epsilon}$.} \]
 Hence, one has a notion of the automorphism group of a Jordan pair $ (J^+, J^-)$, which is a subgroup of $\GL(J^+) \times \GL(J^-)$. 
 Via the projection onto each of the factors, $\Aut(J^+, J^-)$ has two natural representations. 
 \vskip 5pt

 \vskip 5pt
 \subsection{\bf Examples}
 We give two pertinent examples here:
 \vskip 5pt
\begin{itemize}
\item For a vector space $V$,  set
\[  (V^+, V^-) = (V, V^*), \]
 and define:
 \[  Q^+(v) ( v^*) = \langle v, v^* \rangle \cdot v, \qquad  Q^-(v^*) (v) = \langle v^*, v \rangle \cdot v^*. \]
 This defines a Jordan pair whose automorphism group is the subgroup
 \[  \GL(V) \hookrightarrow \GL(V^+) \times \GL(V^-). \]
  Thus, this gives a description of the general linear group without favouring one of its two standard representations. 
   \vskip 5pt
   
   \item For a cubic norm structure $J$, the pair 
   \[  (J^+, J^-) = (J, J^*) \]
   inherits the structure of a Jordan pair from its Jordan algebra structure  and the automorphism group of $(J^+, J^-)$ is precisely the structure group  of $J$.  More precisely,  
   recall that in addition to the cubic norm form $N_J$ and an identity element $1_J$, a cubic norm structure $J$ comes equipped with
   \vskip 5pt
   \begin{itemize}
   \item[-]     a nondegenerate symmetric bilinear trace form $T: J \times J \rightarrow F$;
   \item[-]   a quadratic map $x \mapsto x^{\#}$ from $J$ to itself, with linearization 
   \[ x \times y = (x+y)^{\#} - x^{\#} - y^{\#}. \]
   \end{itemize}
   Given these, one sets
   \[  U_x(y) =  T(x,y)x - x^{\#} \times y \quad \text{for $x,y \in J$} \]
   and  observes that $U$ is quadratic in $x$ and linear in $y$. 
   The reader familiar with the notion of quadratic Jordan algebras will recognize that this is the $U$-operator in that theory.
   One then has a Jordan pair defined by setting:
   \[  (J^+, J^-) = (J, J) \quad \text{and} \quad (Q^+, Q^-) = (U, U). \]
   Since $T$ gives an identification of $J$ with $J^*$, we may also describe this Jordan pair as $(J, J^*)$. 
   \end{itemize} 
   \vskip 10pt
   
   \subsection{\bf Dual pairs}
 After this brief sidetrack, we can now introduce a family of dual pairs in exceptional group of the form \cite{MS, Ru}
 \[   i: \SL_3 \times \mathrm{iso}(J, N)  \longrightarrow \mathcal{E}  \]
 where $J$ is a Freudenthal Jordan algebra. For split groups, 
 this dual pair is obtained by removing from the extended Dynkin diagram of $\mathcal{E}$ (of type $F_4$ or $E_n$) the simple vertex joined to the unique vertex attached to the extra vertex. 
 One sees that the extended Dynkin diagram breaks into two pieces, with one of them of type $A_2$.
 \vskip 5pt
 
  By the notion of Jordan pairs introduced above, we recognize that this dual pair is of the form
  \[   \Aut(V^+, V^-)^{der} \times \Aut(J^+, J^-)^{der}  = \SL(V) \times  \mathrm{iso}(J, N_J), \]   
  where $\dim V = 3$ and the superscript {\em der} signifies the derived group.
  \vskip 10pt

 The following summarizes the algebras and groups which occur, where $J(C)$ the space of $3$ by $3$ hermitian symmetric matrices with coefficients in 
 a composition algebra $C$.  
 \vskip 5pt
 
 \begin{center}
 \begin{tabular}{|c|c|c|c|c|}
 \hline
 $\mathcal{E}$  & $F_4$ & $E_6\rtimes \mathrm{Gal}(K/F)$  &$E_7$  &  $E_8$ \\
 \hline
 $C$ & $F$ & quadratic $K$ & quaternion $B$ & octonion $\mathbb{O}$  \\
 \hline
 $\dim J(C)$ & 6 &  9 & 15 & 27  \\
 \hline
 ${\mathrm{iso}}(J,\det)$ & $\SL_3(F)$ & $\SL_3(K)/ {\Res}_{K/F}^1(\mu_3)\rtimes \mathrm{Gal}(K/F) $ & $\SL_3(B)/\mu_2 $ & $E_6^{sc}$ \\
 \hline
 \end{tabular}
 \end{center}
\vskip 5pt

 Observe that the centers of $\SL(V)$ and $\mathrm{iso}(J,\det)$ are isomorphic to $\mu_3$ in all cases and 
  \[  \Ker(i) =  \mu_3^{\nabla}. \]
  Hence, we take $T = \G_m$, so that
  \[  \tilde{J}^{\sim} = ( \GL(V) \times \sim(J, N_J))^{\sim}. \]
  
  \begin{prop} If $J=J(C)$ is  the space of $3$ by $3$ hermitian symmetric matrices with coefficients in 
 a composition algebra $C$, and $N_J=\det$, then the similitude character $\sim :  \sim(J, N_J))(F) \rightarrow F^{\times}$ is surjective.  
 \end{prop} 
 \vskip 5pt
 
 \begin{proof} Let $t\in F^{\times}$.  Let $g: J \rightarrow J$ be the linear transformation such that for any $x\in J$, $y=g(x)$ is obtained from $x$ by rescaling the 
 entries in the following way. The entries in the upper-left  $2\times 2$ block are multiplied by $t$, the remaining diagonal entry is divided by $t$ and the 
 other entries are left unchanged. Then $\det(y)= t \det(x)$.  
 \end{proof} 
 Thus in this case $\tilde G(F)=\tilde G(F)^+$ and $\tilde H(F)=\tilde H(F)^+$
  and the resulting similitude dual pair is 
  \[   \Aut(V^+, V^-) \times \Aut(J^+, J^-) = \GL(V) \times  {\sim}(J, N_J). \]   
  Moreover, the group ${\sim}(J, N_J)(F)$ is given as follows:
  \[ {\sim}(J, N_J)(F) = 
  \begin{cases}
  \GL_3(F), &\text{  if $C = F$;} \\
  (\GL_3(K) \times F^{\times}) / \{ (t, N_{K/F}(t)^{-1}):  t \in K^{\times} \} \rtimes \mathrm{Gal}(K/F), &\text{  if $C = K$;} \\
  (\GL_3(B) \times F^{\times}) / \{ (t, t^{-2}) : t \in F^{\times} \}, &\text{  if $C = B$;} \\
 \mathrm{GE}_6(F), &\text{  if $C = \mathbb{O}$.} 
  \end{cases} \]
     \vskip 5pt

  \subsection{\bf Twisted Jordan pairs}
   It is in fact better to slightly repackage the definition of Jordan pairs introduced above, by setting
 \[  J^{\square} = J^+ \times J^- \quad \text{ regarded as an $F \times F$-module} \]
 and
 \[  Q^{\square}  = Q^+ \times Q^- : J^{\square} \longrightarrow  \Hom_{F \times F} ((J^{\square})^{\sigma}, J^{\square}) \]
 where $\sigma$ is the nontrivial automorphism of the $F$-algebra $F \times F$ given by switching the two components. The axioms (JP1-3) for a Jordan pair can be accordingly reformulated in terms of $(J^{\square},Q^{\square})$. A homomorphism $\phi: (J^{\square},Q_J^{\square}) \longrightarrow (V^{\square}, Q_V^{\square})$ is then a $F \times F$-module homomorphism $J^{\square} \longrightarrow V^{\square}$  such that
 \[  \phi(Q_J^{\square}(x)(y)) = Q_V^{\square}(\phi(x))(\phi(y)) \quad \text{ for $x, y \in J^{\square}$.} \]
 The advantage of such a repackaging is that it allows one to introduce twisted version (or $F$-rational forms) of Jordan pairs, where one replaces $F \times F$ by a separable quadratic field extension $K/F$.     
 \vskip 5pt
 
 Moreover, in the definition of an automorphism of $(J^{\square},Q^{\square})$, we may consider those $\phi$'s which are  $(F \times F, \sigma)$-linear instead of $(F \times F)$-linear.  This gives rise to a larger $F$-automorphism group  $\Aut_F(J^{\square}, Q^{\square})$  containing the subgroup  $\Aut_{F \times F}(J^{\square},Q^{\square})$ as a normal subgroup of index $2$; these extra $F$-automorphisms are outer automorphisms of $\Aut_{F \times F}(J^{\square},Q^{\square})$ and their action on $J^{\square}$ exchanges the two factors $J^+$ and $J^-$. 
  
 \vskip 5pt

 Now let $K$ be a separable \'etale $F$-algebra with nontrivial automorphism $\sigma \in \Aut(K/F)$. We shall define a notion of Jordan pair $(J,Q)$ with respect to $K/F$, following a paper of de Medts \cite{dM} where it was introduced under the guise of  {\em Hermitian cubic norm structure} \cite[\S 4]{dM}.  Such a Jordan pair consists of:
 \vskip 5pt
 \begin{itemize}
 \item a $K$-vector space $J$;
 \item a $K$-quadratic map $Q: J \longrightarrow \Hom_K(J^{\sigma}, J)$
 \end{itemize}
 satisfying the reformulated analog of (JP1-3) (we will not elaborate further here).   Let us give two examples:
\vskip 5pt

\begin{itemize}
\item Suppose $(V, h)$ is a $K$-vector space  equipped with a Hermitian form $h : V \times V^{\sigma} \rightarrow K$. Then defining $Q$ by
\[  Q(v) (w) = h(v,w) \cdot v  \quad \text{for $v,w \in V$} \]
gives a Jordan pair relative to $K/F$.  The automorphism group of $(V,Q)$ is precisely the unitary group $\U(V,h)$. 
\vskip 5pt

\item  In \cite[Thm. 4.6]{dM}, it was explained how a cubic norm structure $J$ over $K$ and a ``$\sigma$-linear  self adjoint autotopy" of $J$ give rise to a Hermitian cubic norm structure on $J$, from which one can deduce a Jordan pair $(J,Q)$ over $K/F$ with $Q$ given by the $U$-operator for the Hermitian cubic norm structure. If $J$ is the exceptional Jordan algebra, then $\Aut_K(J, Q)^{\der}$ is  the quasi-split simply connected $E_6^K$ associated to $K/F$ and 
\[  \Aut_K(J, Q)/ \Aut_K(J, Q)^{\der} \simeq {\Res}^1_{K/F}(\mathbb{G}_m) \]
is the anisotropic 1-dimensional torus associated to $K$.
 \end{itemize}
\vskip 5pt

Associated to these two examples and as a twisted version of the dual $\SL_3 \times \mathrm{iso}(J, N_J)$ introduced in the previous subsection, one has the quasi-split but non-split dual pair
\[   \SU(V,h) \times \Aut_K(J,Q)^{\der} \longrightarrow \mathcal{E}, \]
where we assume for simplicity that $\mathcal{E}$ is split.
The centers of these groups are isomorphic to 
\[  Z = {\Res}^1_{K/F}(\mu_3). \] 
Following Example (3) in \S \ref{S:simi group}, we choose $T = \mathrm{Res}_{K/F}(\G_m)$ and note that there is a short exact sequence
\[ \begin{CD}
1 @>>> Z @>>>  T @>>> T @>>> 1 \end{CD} \]
defined by the map
\[ x \mapsto x^{3}/N_{K/F}(x) \quad \text{ on $T$.} \]
The  associated similitude dual pair is:
\[ \GU(V,h) \times \mathrm{G}\Aut_K(J,Q)  \]
where
\[  \mathrm{G}\Aut_K(V,Q)   = ({\Res}_{K/F}(\mathbb{G}_m) \times \Aut_K(J,Q)^{der}) / Z^{\nabla} = (\mathbb{G}_m \times \Aut_K(J,Q)) / \mu_2^{\Delta}. \]
We note here that the similitude character on $\GU(V,h)$ is not the usual similitude character $\sim: \GU(V,h) \longrightarrow \G_m$. Rather, by construction, it is the homomorphism
\[  \GU(V,h) \longrightarrow \mathrm{Res}_{K/F}(\G_m) \]
given by
\[   g \mapsto \det(g) / \sim(g). \]
This is surjective onto $K^{\times}$, as one can see by restricting it to a maximal torus in a Borel subgroup (or by observing that $\mathrm{SU}(V,h)$ is simply-connected when $F$ is nonarchimedean). 
Hence $\mathrm{G}\Aut_K(J,Q)(F)^+ = \mathrm{G}\Aut_K(J,Q)(F)$.  We will leave it to the reader to work out what $\GU(V,h)(F)^+$ is  in the various cases.
 \vskip 5pt

\section{\bf Seesaw Duality}
In the theory of theta correspondence, seesaw dual pairs and the associated seesaw dualities serve as useful tools. In this section, we examine how this seesaw duality is impacted by the extension to similitude dual pairs. Hence, suppose one has a seesaw diagram of dual pairs in $\mathcal{E}$:
 \[
 \xymatrix{
  G' \ar@{-}[dr] \ar@{-}[d] & H  \ar@{-}[d] \\
  G \ar@{-}[ur] & H'}
\]  
For $\pi \in \Irr (G(F))$ and $\sigma \in \Irr (H'(F))$, the associated seesaw identity reads:
\[ \Hom_{H'}(\Theta(\pi), \sigma) \simeq \Hom_{G \times H'}(\Omega, \pi \otimes \sigma) \simeq \Hom_G( \Theta(\sigma), \pi). \]
\vskip 5pt

Now let us make the following hypotheses:
\vskip 5pt

\begin{itemize}
\item The pair $(G,H)$ gives rise to a corresponding similitude dual pair $(\tilde{G} , \tilde{H})$, as we explained earlier in this paper;
\item for the other dual pair $(G', H')$, $G' \times H'$  maps injectively  into $\mathcal{E}$ (so that there is no associated similitude pair). 
\end{itemize}
Observe also that 
\[  G \times H' \subset \tilde{J}^{\sim} = (\tilde{G} \times \tilde{H})^{\sim}(F). \]
Then we have:

\vskip 5pt

\begin{prop}
Assume the above hypotheses. Then for $\tilde{\pi} \in \Irr (\tilde{G}(F))$ and $\sigma \in \Irr( H'(F))$, one has the similitude seesaw identity:
\[ \Hom_G(\Theta(\sigma), \tilde{\pi}) \simeq   \Hom_{H'}(\Theta(\tilde{\pi}), \sigma). \]
\end{prop}
\vskip 5pt

\begin{proof}
For $\tilde{\pi} \in \Irr (\tilde{G}(F))$ and $\sigma \in \Irr( H'(F))$, we consider
\[  
\Hom_{G \times H'}(\Omega,  \tilde{\pi} \otimes \sigma). \]
On the one hand, this is isomorphic to
\[  \Hom_{G \times H'}(\Omega \otimes \sigma^{\vee}, \tilde{\pi}) = \Hom_G((\Omega \otimes \sigma^{\vee})_{H'}, \tilde{\pi}) \simeq  \Hom_G(\Theta(\sigma), \tilde{\pi}). \]
On the other hand, by Lemma \ref{L:element}(i), we also get
\[  \Hom_{G \times H'} (\Omega \otimes \tilde{\pi}^{\vee}, \sigma) = \Hom_{H'}((\Omega \otimes \tilde{\pi}^{\vee})_G, \sigma) = \Hom_{H'}(\Theta(\tilde{\pi}), \sigma). \]
Hence, one deduces that
\[ \Hom_G(\Theta(\sigma), \tilde{\pi}) \simeq   \Hom_{H'}(\Theta(\tilde{\pi}), \sigma), \]
as desired.
\end{proof}
\noindent In other words, there is no significant problem in extending the seesaw identity to the similitude setting, at least under the hypotheses  in the proposition.
\vskip 5pt

Let us give some examples of the above situation, using the exceptional dual pairs we discussed above. A first example is the seesaw diagram:
\vskip 5pt
  \[
 \xymatrix{
   \Aut(\mathbb{O}) = G_2 \ar@{-}[dr] \ar@{-}[d] &   \mathrm{iso}(J,N)  \ar@{-}[d] \\
 \Aut(V^{\square}, Q^{\square})^{der} = \SL_3    \ar@{-}[ur] &     \Aut(J) }
\]  
where 
\begin{itemize}
\item $J$ is a cubic norm structure over $F$ with norm form $N$;
\item $V$ is a 3-dimensional vector space over $F$, giving rise to a Jordan pair $(V^{\square}, Q^{\square})$;
\item  $\mathbb{O}$ is the octonion $F$-algebra  constructed as $\mathbb{O} = F^2 \times V^{\square}$ (as a structurable algebra \cite[\S 4]{dM}, a notion we did not introduce).
\end{itemize}
We have explained that the dual pair $\Aut(J) \times \Aut(\mathbb{O})$ does not have a similitude version, whereas the other pair $\SL(V) \times \mathrm{iso}(J,N)$ does. 
 \vskip 5pt
 
 As another example, we have:
  \[
 \xymatrix{
\Aut(E \oplus C^{\flat})     \ar@{-}[dr] \ar@{-}[d] &   \Aut_E(\mathbb{O}^{\flat}) = \Spin^E_8   \ar@{-}[d] \\
  \Aut_E(C^{\flat} )\ar@{-}[ur] & \Aut(\mathbb{O}) = G_2}
\]  
 where
 \begin{itemize}
 \item $\mathbb{O}$ is the octonion $F$-algebra giving rise to the twisted octonion algebra $\mathbb{O}^{\flat} = \mathbb{O} \otimes_F E$  over $E/F$;
 \item $C_1^{\flat}$ is a twisted composition algebra relative to $E/F$ giving rise to a a cubic norm structure (or Freudenthal-Jordan algebra) $J = E \oplus C_1^{\flat}$  via the Springer decomposition \cite[\S 38A, Pg. 522]{KMRT}.
 \end{itemize}
 As above, the dual pair $\Aut(\mathbb{O}) \times \Aut(J)$ does not have a similitude version whereas the pair $\Aut(\mathbb{O}^{\flat}) \times \Aut_E(C^{\flat})$ does. 
 \vskip 5pt
 
 Observe that in forming both these seesaw diagrams, the initial data consists of giving two algebraic structures on the bottom row of the diagram, which then induces the algebraic structures in the top row.  Moreover, observe that the two seesaw diagrams can be combined into a single one involving 3 dual pairs.
\vskip 10pt

\section{\bf Global Theta Correspondence}
In this section, we consider the global theta correspondence. Hence, let $k$ be a number field with adele ring $\A$. Suppose that we have a dual pair
\[  i: G \times  H \longrightarrow \mathcal{E} \]
as in the introduction, with $\Ker(i) = Z^{\nabla}$. 
\vskip 5pt

Let $\Omega$ be the global minimal representation of $\mathcal{E}$ and suppose one has an automorphic realization
\[  \theta:  \Omega \longrightarrow \mathcal{A}(\mathcal{E}). \] 
For an irreducible  cuspidal representation $\pi \subset \mathcal{A}_{cusp}(G)$ of $G$, one has the usual notion of global theta lifting. More precisely, for $\phi \in \Omega$ and $f \in \pi$, one sets
\[  \theta(\phi, f)(h) = \int_{[G]} \theta(\phi)(gh) \cdot \overline{f(g)} \, dg \]
for $h \in H(\A)$. Then the global theta lift of $\pi$ is the $H(\A)$-submodule
\[  \Theta(\pi) = \langle \theta(\phi, f): \phi \in \Omega, \, f \in \pi \rangle \subset \mathcal{A}(H). \]
Moreover the map $(\phi, f) \mapsto \theta(\phi, f)$ is a $G(\A)$-invariant and $H(\A)$-equivariant map
\[  \theta:  \Omega \otimes \overline{\pi} \longrightarrow \Theta(\pi). \]
\vskip 5pt

Now suppose we have chosen data (a) and (b) as in \S \ref{S:simi group} and hence have the similitude groups $\tilde{G}$ and $\tilde{H}$. 
On taking points, we have the groups
\[  
\tilde{G}(k)^+  \subset \tilde{G}(\A)^+ = \prod_v \tilde{G}(k_v)^+ \]
and likewise for $\tilde{H}(k)^+ \subset \tilde{H}(\A)^+$.
Moreover, by construction, one has a group homomorphism
\[  \iota:  \tilde{J}^{\sim} = (\tilde{G} \times \tilde{H})^{\sim} \longrightarrow \mathcal{E}.  \]
\vskip 5pt

Let $\tilde{\pi}$ be an irreducible  cuspidal representation of  $\tilde{G}(\A)^+ = \prod_v \tilde{G}(F_v)^+$ with a realization
\[  \tilde{\pi} \subset \mathcal{A}_{cusp}(\tilde{G}(k)^+ \backslash \tilde{G}(\A)^+). \]
 We would like to define its similitude global theta lifting to the space $\mathcal{A}(\tilde{H}^+)$ of automorphic forms on $\tilde{H}(\A)^+$. 
For $f \in \tilde{\pi}$, $\phi \in \Omega$ and $y \in \tilde{J}^{\sim}(\A)$, we set
\[  \theta(\phi, f)(y) = \int_{[G]}  \theta(\phi) ( g \cdot \iota( y)) \cdot \overline{f(g)} \, dg.\]
This function of $\tilde{J}^{\sim}(\A)$ descends to a function on $\tilde{H}(\A)^+$ which is left invariant under $\tilde{H}(k)^+$, thus giving an element of $\mathcal{A}(\tilde{H}^+)$.
 
\vskip 5pt

A more concrete way of defining the function $\theta(\phi,f)$ as a function on $\tilde{H}(\A)^+$ is as follows. Given $h \in \tilde{H}(\A)^+$, one can find an element $t(h) \in \tilde{G}(\A)^+$ such that
\[  \sim_{\tilde{G}}(t(h)) = \sim_{\tilde{H}}(h). \]   
Then one has  an element
\[  (t(h), h) \in \tilde{J}(\A)^{\sim} = (\tilde{G} \times \tilde{H})^{\sim}(\A). \] 
Then 
\[   \theta(\phi, f)(h) = \int_{[G]}  \theta(\phi) ( \iota(g t(h), h)) \cdot \overline{f(g)} \, dg.\]
Since this is merely an explication of the more natural definition of $\theta(\phi,f)$ given initially, the choice of the element $t(h)$ is immaterial.
\vskip 5pt

The span of the functions $\theta(\phi,f)$, as $\phi \in \Omega$ and $f \in \pi$ vary, is the global similitude theta lift $\Theta(\pi)$ of $\pi$. 
We note:
\vskip 5pt

\begin{prop}
Let $\pi \subset \mathcal{A}(\tilde{G}^+)$ be an irreducible  cuspidal representation of $\tilde{G}(\A)^+$. Assume that
\vskip 5pt
\begin{itemize}
\item $\Theta(\pi) \subset \mathcal{A}_2(\tilde{H}^+)$ is nonzero and contained in the space of square-integrable automorphic forms (with a fixed central character);
\item  the local Howe duality property holds at all places $v$ of $k$ for $G(k_v) \times H(k_v)$ or equivalently $\tilde{G}(k_v)^+ \times \tilde{H}(k_v)^+$;
\end{itemize}
Then 
\[  \Theta(\pi) \simeq \bigotimes_v \theta(\pi_v). \]
In particular, $\Theta(\pi)$  is irreducible and one has the compatibility of the local and global similitude theta correspondences.
\end{prop}

\vskip 10pt
\section{\bf Appendix: Clifford theory for $p$-adic groups}  
Let $G$ be a locally compact group and $N$ be a normal open subgroup of $G$ of finite index. Let $A=G/N$, and $m=|A|$.  
We assume that the topology on $N$ (and hence on $G$) is defined by a sequence of 
open compact groups. We assume that $G/K$ is countable for one, and hence all of those open compact groups. 
This condition assures that Schur's lemma holds for smooth irreducible representations of $G$ and $N$. The following elementary result is 
based on \cite[\S 2.7]{BH}.
\vskip 5pt

\begin{prop} \label{P:A1}
Let $(\pi, V)$ be a smooth finite length representation of $G$.  Then 
\begin{enumerate} 
\item $V$ is a finite length $N$-module. 
\item $V$ is semisimple $G$-module if and only if it is a semisimple $N$-module.  
\item The $N$-socle of $V$ is equal to the $G$-socle of $V$. 
\end{enumerate} 
\end{prop} 
\begin{proof}  (1) We can assume that $V$ is irreducible. Since $V$ is finitely generated over $N$, it admits an irreducible $N$-quotient $U$.  Then, by Frobenius reciprocity, 
 $V$ is a submodule of $\Ind_N^G U$. The latter is a semi-simple $N$-module of finite length.  Thus $V$ is not only finite length but also a semi-simple $N$-module. 
  This completes (1) and gives one direction of (2).  Now assume that $V$ is $N$-semisimple. 
Let $W\subset V$ be a $G$-submodule. Since $V$ is $N$-semisimple, there exists an $N$-invariant projection $P$ of $V$ onto $W$. Then 
\[ 
\frac{1}{m}\sum_{g\in G/N} \pi(g) P\pi(g)^{-1} 
\] 
is a $G$-invariant projection of $V$ onto $W$.  Let $U\subset V$ be the kernel of this projection. Then $V=W\oplus U$ and $V$ is $G$-semisimple.  (3) The $G$-socle of $V$ is $H$-semisimple by 
(2) hence it is contained in the $H$-socle. On the other hand, the $H$-socle is $G$-stable, and $G$-semisimple by (2). Hence it is contained in the $G$-socle. 
\end{proof} 

Next we want to describe smooth irreducible representations of $G$ in terms of those of $N$. 
Let $V$ be an irreducible representation of $G$ and $U$ an irreducible 
$N$-quotient.  By Frobenius reciprocity, $V$ is an irreducible submodule of $\Ind_N^G(U)$. Thus we need to analyze the induction from $N$ to $G$. 
For every $g\in G$, let $U^g$ be the $g$-conjugate of $(\pi, U)$, that is,  the representation of $N$ on $U$ where every $n\in N$ acts by  $\pi(gng^{-1})$. 
Let $G_U$ be the stabilizer of $U$ in $G$, that is, the subgroup of $G$ consisting of all $g\in G$ such that $U$ is isomorphic to $U^g$, as $N$-modules.  

\begin{lemma}  Let $V$ be an irreducible representation of $G_U$ containing $U$. Then 
\[ 
\Ind_{G_U}^G V 
\] 
is an irreducible representation of $G$.
\end{lemma} 
\begin{proof}  Observe that the restriction of $V$ to $N$ is a multiple of $U$. Let $U_1, \ldots , U_l$ be all non-isomorphic $G$-conjugates of $U$.  We have a decomposition, as an $N$-module, 
\[ 
\Ind_{G_U}^G V  = \bigoplus_{i=1}^l V_i  
\] 
where $V_i$ is a multiple of $U_i$.  Observe that $G$ permutes these summands, and that each $V_i$ is an irreducible $G_{U_i}$-module.  Furthermore, for each $i$, 
\[ 
V_i'= \bigoplus_{j\neq i} V_j  
\] 
is a $G_{U_i}$-module.  Now let $W$ be a $G$-submodule of $\Ind_{G_U}^G V$. Fix an $i$. It is easy to see that any $G_i$-submodule of $\Ind_{G_U}^G V$, in particular $W$, 
 either contains $V_i$ or is contained in $V_i'$.  Since $W$ is a $G$-module, it follows that $W=\Ind_{G_U}^G V$ or $0$.  
\end{proof}

Next we need to understand irreducible representation of $G_U$ containing $U$. To simplify notation assume that $G=G_U$.  
Then, for every $g\in G$, there exists $A_g : U \rightarrow U$ that intertwines $U$ and $U^g$.  By Schur's lemma $A_g$  
is unique up to an element in $\mathbb C^{\times}$.  Thus $U$ is naturally a module for a  central extension 
\[ 
1\rightarrow \mathbb C^{\times} \rightarrow \tilde G \rightarrow G \rightarrow 1. 
\] 
This extension splits over $N$, indeed, we have a canonical choice $A_n=\pi(n)$ for all $n\in N$. After taking the quotient by $N$ we get a central extension 
\[ 
1\rightarrow \mathbb C^{\times} \rightarrow \tilde A \rightarrow A \rightarrow 1. 
\] 
Next, we have 
\[ 
U\otimes \Ind_{\mathbb C^{\times}}^{\tilde A} (\epsilon)  \simeq \Ind_N^G(U) 
\] 
where $U$ is viewed as $\tilde G$-module and $\epsilon : \mathbb C \rightarrow \mathbb C$ is the inverse character.  This isomorphism is realized by 
\[ 
u \otimes f \mapsto (g \mapsto  f(g)\pi(g)(u)). 
\] 
Thus in order to decompose $\Ind_N^G(U)$ it suffices to decompose $ \Ind_{\mathbb C^{\times}}^{\tilde A} (\epsilon)$. 
\smallskip 

Claim: The extension $\tilde A$ is defined by a co-cycle with values in $\mu_{m}$. To see this, define a section $s : A\rightarrow \tilde A$ such that $s(a)$,  for every $a\in A$, acts on 
$ \Ind_{\mathbb C^{\times}}^{\tilde A} (\epsilon)$ as a linear transformation of determinant 1.  
Then the extension $\tilde A$ is determined by the co-cycle $c(a,b)$ defined by 
\[ 
s(a) s(b) =c(a,b) s(ab). 
\]
After taking determinant of both sides, we get $c(a,b)^m=1$, as claimed. 

\smallskip 
Using well known facts from representations of finite groups, we can now write  
\[ 
\Ind_{\mathbb C^{\times}}^{\tilde A} (\epsilon) \simeq \oplus_{E}  \dim(E)\cdot  E
\]   
where the sum runs over all irreducible representations $E$ of $\tilde A$ such that $\mathbb C^{\times} \subset \tilde A$ acts by $\epsilon$. Call such representations genuine. 

\vskip 5pt

\begin{lemma} Let $U$ be an irreducible representation of $N$ such that $G_U=G$. With notation as above 
\[ 
\Ind_N^G U = \bigoplus_{E}  \dim(E) \cdot U\otimes E 
\] 
where the sum is over all irreducible genuine representations $E$ of $\tilde A$.  
Representations $U\otimes E$ are irreducible and mutually non-isomorphic. 
\end{lemma} 
\begin{proof} It remains to prove the last sentence. Let $E_1$ and $E_2$ be any two genuine $\tilde A$-modules. Since $U$ is irreducible $N$-module,  
\[ 
\Hom_N( U\otimes E_1 , U\otimes E_2) \simeq \Hom_{\mathbb C}(E_1, E_2). 
\] 
Thus 
\[ 
\Hom_G( U\otimes E_1 , U\otimes E_2) \simeq \Hom_{\tilde A}(E_1, E_2). 
\] 
The lemma follows.   
\end{proof} 

We now briefly discuss the lemma when $A$ is abelian. Then $\tilde A$ is a two-step nilpotent group. 
The commutator of elements in $\tilde A$ defines a skew-linear form on $A\times A$ 
\[ 
\langle a, b \rangle = [s(a),s(b)]
\] 
Since the order of $A$ is $m$, the skew form, and therefore the commutator, takes values in $\mu_m$.
The dimension of any genuine irreducible representation $E$ is equal to $\sqrt{|\bar A|}$ where $\bar A$ is the quotient of 
$A$ by the kernel of the skew form $\langle a, b \rangle$. Furthermore, any two irreducible representations of $G$ containing $U$ are  $A$-character twists one of another. Moreover, 
two character twists are isomorphic if and only if the characters coincide on the kernel of the skew form. 

\smallskip 
Finally we combine the two lemmas to obtain the general case. 
\begin{prop} \label{P:A2}
Let $U$ be an irreducible $N$-module, and $G_U$ its stabilizer in $G$. Let $A_U=G_U/N$. Let $\tilde A_U$ be the central extension of $A_U$ arising from the projective 
action of $G_U$ on $U$.  Then 
\[ 
\Ind_N^G U = \bigoplus_{E}  \dim(E) \cdot \Ind_{G_U}^G (U\otimes E) 
\]
where the sum is over all irreducible genuine representations $E$ of $\tilde A_U$.  This is a decomposition of $\Ind_N^G U$ into isotypic summands. 
\end{prop}  
We remark that a study of the type of problems considered in this appendix can be found in the paper \cite{T} of M. Tadi\'c, where it was shown that the restriction of an irreducible representation of $\GL_n(F)$ to $\SL_n(F)$ is multiplicity-free.

\vskip 15pt

\noindent{\bf Acknowledgments:}  This paper is dedicated to Marko Tadi\'c on the occasion of his 70th birthday.
W.T.G. is partially supported by a Singapore government  MOE Tier 1 grant 
R-146-000-320-114 and the Tan Chin Tuan Centennial Professorship at NUS. G. Savin is partially supported by 
 a  National Science Foundation grant DMS-1901745 and by a Simons Foundation gift 946504. The authors would like to thank Skip Garibaldi for pointing them to the notion of Jordan pairs developed in Loos' book \cite{Loos}.

\end{document}